\documentclass[
10pt
]{article}

\usepackage[%
left=0.8in,
right=0.8in,
bottom=0.8in,
top=1in,
footskip=0.2in,
]{geometry}

\usepackage{color}
\definecolor{LinkColor}{rgb}{0,0,1}
\definecolor{lbcolor}{rgb}{0.85,0.85,0.85}
\definecolor{FrameColor}{rgb}{0.85,0.85,0.85}

\usepackage[ngerman, french, english]{babel}
\usepackage[utf8]{inputenc}
\usepackage[T1]{fontenc}
\usepackage{enumerate}
\usepackage[normalem]{ulem}

\def\pskip{\\[-3mm]}

\usepackage[babel,french=guillemets,german=swiss]{csquotes}

\usepackage{graphicx}
\usepackage{caption}
\usepackage{amsmath}
\usepackage{amssymb}
\usepackage{dsfont}
\usepackage{nicefrac}
\usepackage{empheq}
\allowdisplaybreaks
\usepackage{upgreek}

\usepackage[%
pdftitle={Titel},
pdfauthor={Autor},
pdfcreator={LaTeX, LaTeX with hyperref and KOMA-Script},
pdfsubject={Betreff}, 
pdfkeywords={Keywords}
]{hyperref} 

\hypersetup{colorlinks=true,
	linkcolor=LinkColor,%
	anchorcolor=LinkColor,%
	citecolor=red,%
	filecolor=LinkColor,%
	menucolor=LinkColor,%
	urlcolor=LinkColor,%
}

\usepackage{amsthm}

\newtheoremstyle{tstyle}
{15pt}	
{5pt}	
{\itshape}	
{}	
{\bfseries}	
{.}	
{0.5em}	
{}	

\theoremstyle{tstyle}

\newtheorem{thm}{Theorem}[section]

\newtheorem{prop}[thm]{Proposition}
\newtheorem{cor}[thm]{Corollary}
\newtheorem{defn}[thm]{Definition}

\newtheoremstyle{cstyle}
{15pt}	
{5pt}	
{}	
{}	
{\bfseries}	
{}	
{0.2222em}	
{}	

\theoremstyle{cstyle}

\newtheorem*{com}{Comment}

\makeatletter
\g@addto@macro{\thm@space@setup}{\thm@headpunct{}}
\renewenvironment{proof}[1][\proofname]{\par
	\pushQED{\qed}%
	\normalfont \topsep6\p@\@plus6\p@\relax
	\trivlist
	\item[\hskip\labelsep
	\bfseries
	#1\@addpunct{\,}]\ignorespaces
}{%
	\popQED\endtrivlist\@endpefalse
}
\makeatother

\def\RR{\mathbb R}
\def\VV{{\mathcal{V}}}
\def\WW{{\mathcal{W}}}

\def\NN{\mathbb N}
\def\BB{{\mathbb B_K}}
\def\CC{{\mathcal C_{\bar u}}}

\def\IBB{\mathring{\mathbb B}_{K}}
\def\IBBH{\mathring{\mathbb B}_{K/2}}

\def\UU{{{\mathbb U}}}
\def\PP{\mathbb P}
\def\DD{{L^2([0,T];\RR^N)}}

\def\F{{\mathfrak{S}}}
\def\S{{\mathcal{S}}}
\def\A{{\mathcal{A}}}

\def\eps{\varepsilon}
\def\supp{\textnormal{supp\,}}

\def\BR{{B_R(0)}}

\def\dtau{\;\mathrm d\tau}

\def\dx{\;\mathrm dx}
\def\dy{\;\mathrm dy}

\def\dt{\;\mathrm dt}

\def\dxv{\;\mathrm d(x,v)}
\def\dtxv{\;\mathrm d(t,x,v)}

\def\delt{\partial_{t}}
\def\delx{\partial_{x}}
\def\delv{\partial_{v}}

\def\stimes{{\hspace{-0.01cm}\times\hspace{-0.01cm}}}
\def\scdot{{\hspace{1pt}\cdot\hspace{1pt}}}

\def\laplace{\Delta}
\def\divergence{\textnormal{div}}

\def\bigvert{\;\big\vert\;}

\def\tand{\quad\text{and}\quad}
\def\twith{\quad\text{with}\quad}

\def\u{{\bar u}}

\def\tf{{\tilde f}}
\def\th{{\tilde h}}
\def\tg{{\tilde g}}
\def\tu{{\tilde u}}
\def\tB{{\tilde B}}

\def\Lag{{\mathcal{L}}}

\def\wto{\rightharpoonup}

\def\itema{\item[\textnormal{(a)}]}
\def\itemb{\item[\textnormal{(b)}]}

\newcommand{\Underset}[3][0pt]{\ensuremath{\underset{\raise#1\hbox{\small\ensuremath{#2}}}{#3}}}

\begin{document}
\begin{center}	
	\LARGE{\bfseries Optimal control of a Vlasov-Poisson plasma \\ by fixed magnetic field coils}\\[8mm]
	
	\normalsize{Patrik Knopf}\\[2mm]
	\textit{University of Regensburg, 93040 Regensburg, Bavaria, Germany}\\[2mm]
	\texttt{Patrik.Knopf@ur.de}\\[-3mm]
	
	\begin{minipage}[h]{0.275\textwidth}
		\begin{flushright}
			\vspace{-2pt}
			\includegraphics[scale=0.05]{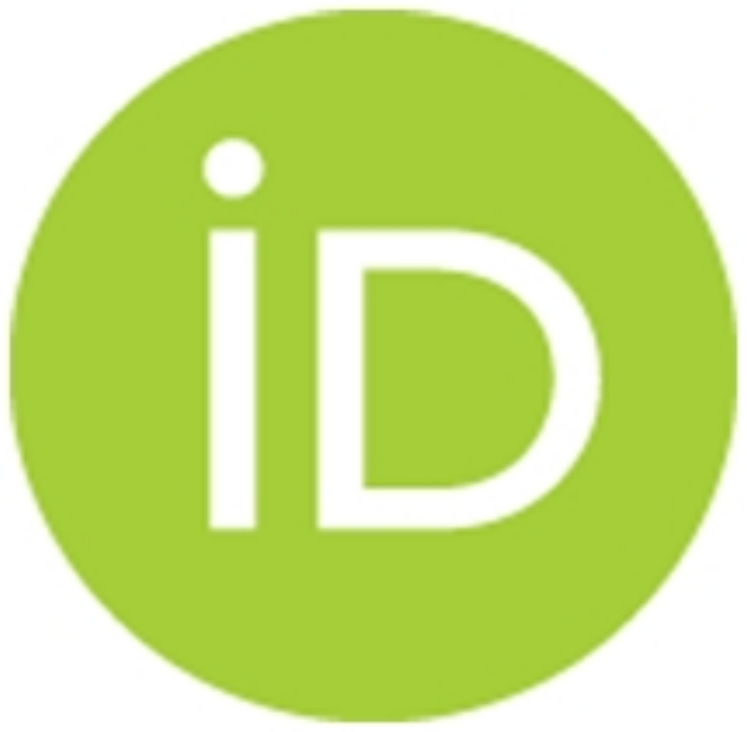} 
		\end{flushright}
	\end{minipage}
	\begin{minipage}[h]{0.5\textwidth}
		\hspace{-12pt}
		\href{https://orcid.org/0000-0003-4115-4885}{orcid.org/0000-0003-4115-4885}
	\end{minipage}
	\vspace{2mm}

	\normalsize{Jörg Weber}\\[2mm]
	\textit{University of Bayreuth, 95440 Bayreuth, Bavaria, Germany}\\[2mm]
	\texttt{Joerg.Weber@uni-bayreuth.de}\\[8mm]
	
	\textit{Please cite as:}  P. Knopf, and J. Weber, Appl Math Optim (2018) \\ \url{https://doi.org/10.1007/s00245-018-9526-5}
	
\end{center}
%
%
\date{\today}
\selectlanguage{english}
\begin{abstract}
	We consider the Vlasov-Poisson system that is equipped with an external magnetic field to describe the time evolution of the distribution function of a plasma. An optimal control problem where the external magnetic field is the control itself has already been investigated in \cite{knopf}. However, in real technical applications it will not be possible to choose the control field in such a general fashion as it will be induced by fixed field coils. In this paper we will use the fundamentals that were established in \cite{knopf} to analyze an optimal control problem where the magnetic field is a superposition of the fields that are generated by $N$ fixed magnetic field coils. Thereby, the aim is to control the plasma in such a way that its distribution function matches a desired distribution function at some certain final time $T$ as closely as possible. This problem will be analyzed with respect to the following topics: Existence of a globally optimal solution, necessary conditions of first order for local optimality, derivation of an optimality system, sufficient conditions of second order for local optimality and uniqueness of the optimal control under certain conditions.\\
	\phantom{}\\
	\textit{Keywords: Vlasov-Poisson equation, optimal control with PDE constraints, nonlinear partial differential equations, calculus of variations.}\\
	\phantom{}\\
	MSC Classification: 49J20, 35Q83, 82D10.
\end{abstract}

%
%
%

\section{Introduction}
The three dimensional Vlasov-Poisson system in the plasma physical case is given by the following nonlinear system of partial differential equations:
\begin{equation}
\label{VP}
\begin{cases}
\partial_t f + v\cdot \partial_x f - \partial_x \psi \cdot \partial_v f = 0,\\[0.15cm]
- \Delta \psi = 4\pi \rho, \quad \lim_{|x|\to\infty} \psi(t,x) = 0,\\[0.15cm]
\rho(t,x) = \int f(t,x,v)\ \mathrm dv.
\end{cases}
\end{equation}
Here $f=f(t,x,v)\ge 0$ denotes the distribution function of the particle ensemble that is a scalar function representing the density in phase space. The time evolution of $f$ is described by the first line of \eqref{VP} which is a first order partial differential equation that is referred to as the Vlasov equation. For any measurable set $M\subset \mathbb R^6$,
$\int_M f(t,x,v)\, \mathrm d(x,v)$
represents the total charge of the particles that have space coordinates $x\in\mathbb R^3$ and velocity coordinates $v\in\mathbb R^3$ with $(x,v)\in M$ at time $t\ge 0$. The function $\psi$ is the electrostatic potential that is induced by the charge of the particles. It is given by Poisson's equation $-\laplace \psi = 4\pi\rho$ with an homogeneous boundary condition where $\rho$ denotes the volume charge density. The self-consistent electric field is then given by $-\partial_x \psi$. Note that both $\psi$ and $-\partial_x \psi$ depend linearly on $f$. Hence the Vlasov-Poisson system is nonlinear due to the term $-\partial_x \psi \cdot \partial_v f$ in the Vlasov equation. Assuming $f$ to be sufficiently regular (e.g., $f(t) := f(t,\cdot,\cdot)\in C^1_c (\mathbb R^6)$ for all $t \ge 0$), we can solve Poisson's equation explicitly and obtain
\begin{equation}
\label{PSIF}
\psi_f(t,x) = \iint \frac{f(t,y,w)}{|x-y|} \;\mathrm dw \mathrm dy, \qquad \delx\psi_f(t,x) = -\iint \frac{x-y}{|x-y|^3} f(t,y,w)\,\mathrm dw \mathrm dy 
\end{equation}
Considering $f\mapsto \psi_f$ (or $f\mapsto -\delx\psi_f$ respectively) to be a linear operator we can formally rewrite the Vlasov-Poisson system as
\begin{equation}
\label{VP2}
\partial_t f + v\cdot \partial_x f - \partial_x \psi_f \cdot \partial_v f = 0.
\end{equation}
The function $f$ can be considered as the state variable. Later, we will additionally introduce a control function. Combined with the condition
\begin{equation}
\label{IC}
f|_{t=0} = \mathring f
\end{equation}
for some function $\mathring f\in C^1_c(\RR^6)$ we obtain an initial value problem. A first local existence and uniqueness result to this initial value problem was proved by R.~Kurth \cite{kurth}. Later J. Batt \cite{batt} established a continuation criterion which claims that a local solution can be extended as long as its velocity support is under control. Finally, two different proofs for global existence of classical solutions were established independently and almost simultaneously, one by K. Pfaffelmoser \cite{pfaffelmoser} and one by P.-L. Lions and B. Perthame \cite{lions-perthame}. Later, a greatly simplified version of Pfaffelmoser's proof was published by J. Schaeffer \cite{schaeffer}. This means that the follwing result is established: Any nonnegative initial datum $\mathring f \in C^1_c(\mathbb R^6)$ launches a global classical solution $f\in C^1([0,\infty[\times\mathbb R^6)$ of the Vlasov-Poisson system \eqref{VP} satisfying the initial condition \eqref{IC}. Moreover, for every time $t\in[0,\infty[$, $f(t)=f(t,\cdot,\cdot)$ is compactly supported in $\mathbb R^6$.
Hence equation \eqref{PSIF} and the reformulation of the Vlasov-Poisson system \eqref{VP2} are well-defined in the case $\mathring f \in C_c^1(\mathbb R^6)$. For more information we recommend to consider the review article \cite{rein} by G. Rein that gives an overview on the most important results.\pskip

To control the distribution function $f$ we add an external magnetic field $B=B(t,x)$ to the Vlasov equation:
\begin{equation}
\label{VPC}
\partial_t f + v\cdot \partial_x f - \partial_x \psi_f \cdot \partial_v f + (v\times B) \cdot \partial_v f = 0, \quad f|_{t=0} = \mathring f\;.
\end{equation}
The cross product $v\times B$ occurs since, unlike the electric field, the magnetic field interacts with the particles via Lorentz force. From a theoretical  point of view, the magnetic field $B$ can be thought of as the control itself. In~\cite{knopf}, an optimal control problem for the state equation \eqref{VPC} was investigated: The basics for calculus of variations were established and a model problem of optimal control with a tracking type cost functional was studied. For the analysis, high regularity assumptions on the control fields have been necessary. This, however, is not desirable in terms of optimal control theory and brings several disadvantages. Especially, the derivation of optimality conditions is restricted by these strong regularity requirements.\pskip

Also, from a technical perspective, it will not be possible to create magnetic control fields ad libitum. Therefore, in this paper, we suppose that the exterior magnetic field is induced by a finite number of fixed field coils. We assume that the current density $J_i$ and the induced magnetic field $B_i$ of the $i$-th coil are given by
\begin{align*}
	J_i(t,x) = j_i(t)\, c_i(x) \tand B_i(t,x) = u_i(t)\, m_i(x).
\end{align*}
with functions $c_i,m_i:\RR^3\to\RR^3$ and $j_i,u_i:\RR\to\RR$. The function $c_i$ corresponds to the geometry of the coil while $j_i$ corresponds to the intensity of the current that flows through the coil. Without loss of generality, we may assume that $c_i$ has compact support. For simplicity, we suppose that the field $B_i$ is determined by the Biot-Savart law which is a magnetostatic approximation of Maxwell's equations. This is a very common approach in physics and engineering science to describe the magnetic field induced by a field coil. We obtain
\begin{align*}
	-u_i(t)\, \nabla\times m_i(x) = -\nabla\times B_i(t,x) = -4\pi\, J_i(t,x) = -4\pi j_i(t)\, c_i(x)
\end{align*}
which is solved by
\begin{align*}
	u_i(t) = j_i(t) \tand m_i(x) = \int\limits_{\RR^3} \frac{c_i(y)\times(x-y)}{|x-y|^3} \dy.
\end{align*}
Note that $m_i$ is source-free (i.e., $\divergence_x(m_i)=0$). Then, due to linear superposition, the total external magnetic field is given by
\begin{align*}
	B(u)(t,x)=\sum_{i=1}^N u_i(t)m_i(x). 
\end{align*}
Now, we will take the vector $u=(u_1,...,u_N)^T$ to be the control in our model. The complete magnetic field will be described by the control-to-field operator $u\mapsto B(u)$. Mathematically, this brings a lot of advantages compared to the model discussed in \cite{knopf}. This is mainly because no stronger regularity assumptions than $u\in L^2([0,T];\RR^N)$ are needed. Therefore, the set of admissible controls can be chosen as a box restricted subset of $L^2([0,T];\RR^N)$ which gives us better possibilities for establishing optimality conditions compared to \cite{knopf}. For instance, we can show that a locally optimal control satisfies a certain projection formula. From this formula we derive the optimality system which in turn can be used to obtain a uniqueness result on small time intervals. Altogether, the model we study in this paper is also more viable for possible numerical implementations. \pskip

To establish the fundamentals of calculus of variations and optimal control for our model, we can basically exploit the theory that was developed in \cite{knopf}. In the next section we will sketch the most important essentials.\pskip

Finally, we want to mention some works on similar topics we could find in literature. The controllability of the Vlasov-Poisson system where the control appears as a right-hand side of the Vlasov equation was studied by O.~Glass and D.~Han~Kwan \cite{glass0,glass1}. They also established a similar result for the relativistic Vlasov-Maxwell system \cite{glass2}. The time evolution and magnetic confinement of Vlasov-Poisson plasmas under the influence of an exterior magnetic field was investigated by S.~Caprino, G.~Cavallaro and C.~Marchioro \cite{caprino3,caprino,caprino2}. For further results on magnetic confinement of plasmas modeled by Vlasov equations see, e.g, \cite{nguyen2,nguyen,skubachevskii}.

\section{Notation, preliminaries and basics}

Our notation is mostly standard or self-explaining. However, to avoid misunderstandings, we fix some of it here.\pskip

Let $d\in\NN$, $U\subset \RR^d$ be any open subset and $k\in \NN$ be arbitrary. $C^k(U)$ denotes the space of $k$ times continuously differentiable functions on $U$, $C^k_c(U)$ denotes the space of $C^k(U)$-functions having compact support in $U$ and $C^k_b(U)$ denotes the space of $C^k(U)$-functions that are bounded with respect to the norm
\begin{align*}
\|u\|_{C^k_b(U)} := \underset{|\alpha|\le k}{\sup} \; \|D^\alpha u\|_\infty = \underset{|\alpha|\le k}{\sup} \; \underset{x\in U}{\sup}\; |D^\alpha u(x)|.
\end{align*}
Note that $\big( C^k_b(U), \|\cdot\|_{C^k_b(U)} \big)$ is a Banach space. \pskip

For any measurable subset $U\subset\RR^d$, $1\le p\le \infty$ and $k\in\NN$, $L^p(U)$ denotes the standard $L^p$-space on $U$ and $W^{k,p}(U)$ denotes the standard Sobolov space on $U$ as, for instance, defined by E. Lieb and M. Loss in \cite[s.\,2.1,6.7]{lieb-loss}. Endowed with their standard norms defined by
\begin{align*}
	\|u\|_{L^p(U)}= \left(\int\limits_U |u(x)|^p \dx \right)^{\frac 1 p} \tand \|u\|_{W^{k,p}(U)} = \left(\sum_{|\alpha|\le k} \|\partial^\alpha u(x)\|_{L^p(U)}^p\right)^{\frac 1 p}
\end{align*}
the spaces $L^p(U)$ and $W^{k,p}(U)$ are Banach spaces. The space $L^2(U)$ is even a Hilbert space with the inner product
\begin{align*}
	\langle u,v \rangle_{L^2(U)} = \int\limits_U u(x)\, v(x) \dx.
\end{align*}
By $L^p(U;\RR^n)$ and $W^{k,p}(U;\RR^n)$ (with $n\in\NN$) we denote the spaces of vector valued functions $u:U\to \RR^n$ whose components lie in $L^p(U)$ or in $W^{k,p}(U)$ respectively. The standard norms on $L^p(U;\RR^n)$ and $W^{k,p}(U;\RR^n)$ are defined analogously where the symbol $|\cdot|$ now stands for the euclidean norm in $\RR^n$.
If $U=\RR^d$ we will often omit the argument "$(\RR^d)$" or "$(\RR^d;\RR^n)$".\pskip

We will also use Banach space valued $L^p$-spaces (so-called Bochner spaces) as defined by L. C. Evans \cite[p.\,301-305]{evans}. For any Banach space $X$ and any real number $T>0$, $L^p(0,T;X)$ denotes the space of Banach space valued $L^p$-functions $[0,T]\ni t\mapsto u(t) \in X$. In this case, $L^p(0,T;X)$ is also a Banach space with the standard norm
\begin{align*}
\|u\|_{L^p(0,T;X)}= \left(\int\limits_0^T \|u\|_X \dx \right)^{\frac 1 p}.
\end{align*}

\begin{prop}
	\label{SAC}
	For $T>0$ and $\beta>3$ we define the Banach space $\WW$ by 
	\begin{gather*}
	\WW:= L^2\big(0,T;W^{2,\beta}(\RR^3;\RR^3)\big) \twith \|\cdot\|_\WW := \|\cdot\|_{L^2(0,T;W^{2,\beta}(\RR^3;\RR^3))}.
	\end{gather*}
	For any $K>0$, the closed ball\vspace{-3mm}
	\begin{align*}
	\BB:=\Big\{ B\in\WW \;\Big\vert\; \|B\|_\WW \le K \Big\}
	\end{align*}
	is referred to as the \textbf{set of admissible fields} (with radius $K$).  Then the following holds:
	\begin{enumerate}
		\item[\textnormal{(a)}]	$\BB$ is a bounded, convex and closed subset of $\WW$.
		\item[\textnormal{(b)}]	$\BB$ is a subset of $L^2\big(0,T;C^{1,\gamma}(\RR^3;\RR^3)\big)$ with $\gamma = 1- \nicefrac 3 \beta$ and there exists some constant $C>0$ that depends only on $\beta$ such that for all $B\in\BB$, \vspace{-1mm}
		$$ \|B(t)\|_{C^{1,\gamma}} \le C\, \|B(t)\|_{W^{2,\beta}} \quad \text{for almost all}\; t\in [0,T].$$
		\item[\textnormal{(c)}]	$\BB$ is a relatively weakly sequentially compact subset of $\WW$.
	\end{enumerate}
\end{prop}

\pagebreak[2]

This result has been established in \cite[Lem.\,3]{knopf} for $\VV:=\WW\cap L^2(0,T;H^1)$ instead of $\WW$. However, in the following approach it will not be necessary that the magnetic fields are in $L^2(0,T;H^1)$ so we will just drop this condition. One can easily see that the proofs of \cite{knopf} that we are referring to hold true in this case.\pskip

It was also proved in \cite[Thm.\,2]{knopf} that any admissible field $B$ has its corresponding unique solution of the initial value problem \eqref{VPC}. However, as these fields are only $L^2$ in time, the solution is not classical but merely strong.

\begin{prop}
	\label{GWS}
	Let $B\in\BB$ be any admissible field and suppose that ${\mathring f \in C^2_c(\RR^6)}$. Then there exists a unique strong solution $f_B$ of the initial value problem \eqref{VPC} to the field $B$, i.e., the following items hold:
	\begin{enumerate}
		\item[\textnormal{(a)}]		$f_B\in W^{1,2}(0,T;C_b(\RR^6))\cap C([0,T];C^1_b(\RR^6))\cap L^\infty(0,T;W^{2,\beta}(\RR^6))$ with
		\begin{gather*}
		\|f_B(t)\|_p = \|\mathring f\|_p\;,\quad t\in [0,T]\;,\quad 1\le p\le \infty,\\
		\|f_B\|_{W^{1,2}(0,T;C_b)} + \|f_B\|_{C([0,T];C^1_b)} + \|f_B\|_{L^\infty(0,T;W^{2,\beta})} \le C
		\end{gather*}
		for some constant $C>0$ depending only on $\mathring f$, $T$, $K$ and $\beta$.
		\item[\textnormal{(b)}]	$f_B$ satisfies the Vlasov equation
		\begin{align*} 
		\delt f_B + v\cdot \partial_x f_B - \partial_x \psi_{f_B} \cdot \partial_v f_B + (v\times B) \cdot \partial_v f_B= 0
		\end{align*}
		almost everywhere on $[0,T]\times\RR^6$.
		\item[\textnormal{(c)}]	$f_B$ satisfies the initial condition $f_B\big\vert_{t=0}=\mathring f$ everywhere on $\RR^6$.
		\item[\textnormal{(d)}]	There exists some radius $R>0$ depending only on $\mathring f$, $T$, $K$ and $\beta$ such that $ \supp f_B(t) \subset \BR$ for all $t\in[0,T]$.
	\end{enumerate}
\end{prop}

\bigskip

With this knowledge it is possible to define an operator that maps any admissible field onto its corresponding solution. The most important attributes of this field-to-state operator have also already been established in \cite{knopf}. We will summarize them in the following proposition:

\begin{prop}
	\label{FSO}
	For any field $B\in\BB$, let $f_B$ denote its corresponding strong solution of the state equation \eqref{VPC}. The operator 
	\begin{align*}
		\F:\BB \to C\big([0,T];L^2(\RR^6)\big), \quad \F(B):=f_B
	\end{align*}
	is called the \textbf{field-to-state operator}. It has the following essential properties:
	\begin{enumerate}
		\item[\textnormal{(a)}]	The field-to-state operator is \textbf{Lipschitz/Hölder continuous} in the following sense: There exists constants $C_L, C_H>0$ depending only on $\mathring f$, $T$, $K$ and $\beta$ such that for all $B,\tB\in\BB$,
		\begin{gather*} 
		\|\F(B) - \F(\tB)\|_{C([0,T];C_b)} \le C_L\, \|B-\tB\|_\WW, \\
		\|\F(B) - \F(\tB)\|_{W^{1,2}(0,T;C_b)} + \|\F(B) - \F(\tB)\|_{C([0,T];C^1_b)} \le C_H\, \|B-\tB\|_\WW^\gamma
		\end{gather*}
		where $\gamma$ is the constant from Proposition \ref{SAC}.
		\item[\textnormal{(b)}]	The field-to-state operator has the following property: Let $(B_k)\subset \BB$ be any sequence with $B_k \wto B \in\BB$ if $k\to\infty$. Then there is a subsequence $(B_{k_j})$ of $(B_k)$ such that
		$$ \F(B_{k_j}) \wto \F(B) \quad \text{in} \quad W^{1,2}\big(0,T;L^2(\RR^6)\big) $$ for $j\to\infty$. Since $\BB$ is weakly compact, any sequence $(B_k)\subset \BB$ actually has a subsequence that converges weakly to some limit $B\in\BB$. Therefore, the weak closure in $W^{1,2}\big(0,T;L^2(\RR^6)\big)$ of the image $\F(\BB)$ is weakly (sequentially) compact. This means that $\F$ is a \textbf{weakly compact} operator.
		\item[\textnormal{(c)}]	The field-to-state operator is \textbf{Fréchet differentiable} on $\BB$. For any field $B\in\BB$ and any direction $H\in\WW$, the Fréchet derivative $\F'(B)[H]=f'_B[H]$ is the unique strong solution of the initial value problem
		\begin{align}
			\label{FDEQ}
			\delt f + v\cdot\delx f - \delx\psi_{f_B}\cdot\delv f - \delx\psi_f\cdot\delv f_B + (v\stimes B)\cdot\delv f + (v\stimes H)\cdot\delv f_B = 0, \quad f\big\vert_{t=0} = 0.
		\end{align}
		This means that $f'_B[H]$ lies in $L^\infty(]0,T[\times\RR) \cap H^1(]0,T[\times\RR) \subset C([0,T];L^2)$, satisfies \eqref{FDEQ} almost everywhere and there exists some radius $\varrho>0$ depending only on $T,K,\mathring f$ and $\beta$ such that $\supp f'_B[H](t) \subset B_\varrho(0)$ for all $t\in[0,T]$.\pskip
		
		Moreover, the Fréchet derivative $\F'(B)[H]$ depends Hölder-continuously on $B\in\BB$ in the following sense: There exists some constant $C>0$ depending only on $\mathring f$, $T$, $K$ and $\beta$ such that for all $B,\tB\in\BB$ and $H\in\VV$,
		$$ \underset{\|H\|_\WW\le 1}{\sup}\; \|\F'(B)[H] - \F'(\tB)[H]\|_{L^2(0,T;L^2)} \le C\, \|B-\tB\|_\WW^\gamma. $$
		where $\gamma$ is the constant from Proposition \ref{SAC}.
	\end{enumerate}
\end{prop}

In \cite{knopf} these properties could finally be used to analyze an optimal control problem with a tracking type cost functional where the field $B$ was the control itself. 


\section{The set of admissible controls and the control-to-state operator}

As motivated in the introduction, we assume that our magnetic field $B$ is induced by $N$ field coils. We assume that each coil generates a magnetic field of a certain shape $m_i=m_i(x)$ and its intensity at time $t$ is determined by a factor $u_i(t)$. This means that the magnetic field of the $i$-th coil is given by $B_i(t,x)=u_i(t)\, m_i(x)$ and thus
\begin{align*}
B(u)(t,x) = \sum_{i=1}^N u_i(t)\, m_i(x).
\end{align*}
is the total external magnetic field. 
The intensity factor $u_i$ is directly proportional to the current that flows through the $i$-th coil. 
Now the vector $(u_1,...,u_N)^T$ is the control in this model and the magnetic field can be interpreted as the function value of the operator $u\mapsto B(u)$. We suppose that $m_i\in W^{2,\beta}(\RR^3;\RR^3)$ for every index $i\in\{1,...,N\}$ and, since real magnetic fields are always source-free, we also assume that $\divergence_x m_i = 0$ on $\RR^3$. The control $u$ is supposed to lie in $L^2([0,T];\RR^N)$ in order to ensure that the field $B(u)$ has the desired regularity. This is specified by the following definition:

\begin{defn}
	Let $N$ be a fixed positive integer and let $M>0$ be a real number. 
	For every $i\in\{1,\,...\,, N\}$ let $m_i=(m_{i1},m_{i2},m_{i3})^T$ be a fixed vector-valued function in $W^{2,\beta}(\RR^3;\RR^3)\subset C^{1,\gamma}(\RR^3;\RR^3)$ with $\|m_i\|_{W^{2,\beta}}\le M$ and $\divergence\, m_i = 0$ on $\RR^3$.
	Moreover let $a=(a_1,...,a_N)^T$ and $b=(b_1,...,b_N)^T$ be fixed functions in $L^2([0,T];\RR^N)$ with $a_i\le 0 \le b_i$ almost everywhere on $[0,T]$ for all ${i\in\{1,...,N\}}$. We define
	\begin{align*}
	&\UU_i:=\Big\{ w \in L^2\big([0,T]\big) \;\Big|\; a_i \le w \le b_i \text{ a.e. on }[0,T]  \Big\},\quad i=1,...,N \tand
	\quad \UU:=\UU_1 \times \;...\; \times \UU_N\;.
	\end{align*}
	The set $\UU$ is referred to as the \textbf{set of admissible controls}. Moreover we define the operator\vspace{-2mm}
	\begin{align*}
	B(\cdot):L^2([0,T];\RR^N) \to L^2\big(0,T;W^{2,\beta}(\RR^3;\RR^3)\big), \; u\mapsto B(u) 
	\quad\text{where}\quad
	B(u)(t,x):=\sum_{i=1}^N u_i(t)\, m_i(x)\;.
	\end{align*}
	The operator $B(\cdot)$  is referred to as the \textbf{control-to-field operator} and 
	\begin{align*}
		\S:\UU\to C\big([0,T];L^2(\RR^6)\big), u\mapsto \S(u):=\F(B(u))
	\end{align*}
	is called the \textbf{control-to-state operator}.
\end{defn}

Note that the set of admissible controls is not empty as the zero function lies in $\UU$. This definition does only make sense if the fields that are generated by the control-to-field operator are admissible in the sense of Proposition~\ref{SAC}. In this case the state $\S(u)=f_{B(u)}$ is well-defined but we also have to investigate how it depends on the control $u$: 
\begin{prop}
	\label{BOP}
	$\;$
	\begin{itemize}
		\item[\textnormal{(a)}] The set $\UU$ is a bounded, convex and closed subset of $L^2([0,T];\RR^N)$ and thus it is relatively weakly sequentially compact. 
		\item[\textnormal{(b)}] The operator $B(\cdot)$ is linear and continuous and there exists some constant $K>0$ depending only on $N,a,b$ and $M$ such that $B(\UU)\subset \IBBH\subset \BB$. This means that the control-to-field operator maps admissible controls onto admissible fields. 
		\item[\textnormal{(c)}] The control-to-field operator $B(\cdot)$ is continuously Fréchet differentiable and its Fréchet derivative at the point $u\in\DD$ is given by
		\begin{align*}
		B'(u)[h] = B(h) \quad\text{for all}\quad h\in L^2([0,T];\RR^N)\,.
		\end{align*}
		\item[\textnormal{(d)}] The control-to-state operator $\S$ is Fréchet differentiable on $\UU$ and its Fréchet derivative at the point $u\in\UU$ is given by
		\begin{align*}
			\S'(u)[h]=\frac{\mathrm d \F\big(B(u)\big)}{\mathrm du}[h] = \F'\big(B(u)\big)[B(h)] \quad\text{for all}\quad h\in \DD\,.
		\end{align*}
		Recall that $\F'\big(B(u)\big)[B(h)]$ is determined by the initial value problem \eqref{FDEQ}. The Fréchet derivative depends Hölder-continuously on $u$, i.e., there exists some constant $C>0$ depending only on $\mathring f, T, K$ and $\beta$ such that 
		\begin{align*}
		\|\S'(u_1) - \S'(u_2)\|_{L^2(0,T;L^2)} \le C\; \|u_1-u_2\|_{L^2([0,T];\RR^N)}^\gamma
		\end{align*}
		for all $u_1,u_2\in\UU$ and $h\in \DD$ where $\gamma$ is the constant from Proposition \ref{SAC}. 
	\end{itemize}
\end{prop}

\noindent Sometimes, for $u\in\UU$ and $h\in \DD$, we will also write $f_u := f_{B(u)}$ to denote $\S(u)$ and $f'_u[h]$ to denote the Fréchet derivative $\S'(u)[h]$.

\begin{proof}
	For any $i\in\{1,...,N\}$ the set $\UU_i\subset L^2([0,T])$ is evidently bounded, convex and closed. Thus weak compactness follows directly from the theorems of Banach-Alaoglu and Mazur. The same holds for the set $\UU\subset \DD$ which proves (a). The operator $B(\cdot)$ is obviously linear and for all $u\in L^2([0,T];\RR^N)$,
	\begin{align*}
	\|B(u)\|_{\WW} &\le \sum_{i=1}^N  \|u_i\|_{L^2([0,T])}  \, \|m_{i}\|_{W^{2,\beta}} \le M\; \sum_{i=1}^N \|u_i\|_{L^2([0,T])}\\
	&\le  M\sqrt{N} \left(\sum_{i=1}^N \|u_i\|_{L^2([0,T])}^2 \right)^{1/2} = M\sqrt{N}\; \|u\|_{L^2([0,T];\RR^N)}\,.
	\end{align*}
	Hence $B(\cdot)$ is continuous. Moreover this yields
	\begin{align*}
	\|B(u)\|_{L^2(0,T;W^{2,\beta})} < M\sqrt{N}\, \big( \|a\|_{L^2([0,T];\RR^N)} + \|b\|_{L^2([0,T];\RR^N)} \big) =:\frac K 2, \quad u\in\UU
	\end{align*}
	and thus $B(\UU)\subset\IBBH$. This proves (b) which directly implies (c). Finally (d) follows directly from (b), (c), Proposition \ref{FSO}(c) and the chain rule.
\end{proof}

\section{The optimal control problem}

We will now suppose that $\lambda_i \ge 0$ for all $i\in\{1,...,N\}$ and $\mathring f, f_d\in C^2_c(\RR^6)$ such that $\|\mathring f\|_{p} = \|f_d\|_{p}$ for all $p\in [1,\infty]$. The aim is to find a control $u\in\UU$ such that the distribution function $f$ matches the desired distribution function $f_d$ as closely as possible. This is modeled by the following minimization problem:
\begin{align}
\begin{aligned}
\text{Minimize}\quad  	&I(f,u) = \frac 1 2 \|f(T)-f_d\|_{L^2(\RR^6)}^2 + \sum_{i=1}^N \frac{\lambda_i}{2}  \|u_i\|_{L^2([0,T])}^2\\
\text{s.t.}\quad 		&\bullet\quad u\;\text{is an admissible control, i.e.,}\;u\in\UU \\
&\bullet\quad f \text{ is a strong solution of the Vlasov-Poisson system} \\
&\qquad \delt f + v\cdot \delx f - \delx\psi_f \cdot\delv f + \big(v\times B(u)\big) \cdot \delv f = 0, \quad f\big\vert_{t=0} = \mathring f  \\
&\qquad\text{to the field}\; B(u),\;\text{i.e.},\; f=\S(u).
\end{aligned}
\end{align}
Recall that $B(u)$ always lies in $\IBBH\subset\BB$ if the control $u$ is admissible. This means that the term \glqq strong solution to the field $B(u)$\grqq\, is well defined. Using the control-to-state operator, this problem can be reduced to
\begin{align}
\label{OP2}
\begin{aligned}
\text{Minimize}\;  	&J(u) = \frac 1 2 \|\S_T(u)-f_d\|_{L^2(\RR^6)}^2 + \sum_{i=1}^N \frac{\lambda_i}{2}  \|u_i\|_{L^2([0,T])}^2, \; 
\text{s.t.}\;		u\in\UU.
\end{aligned}
\end{align}
where $\S_T(u)=f_u(T)$ denotes the control-to-state operator evaluated at time $T$.

\subsection{Existence of a globally optimal solution}

Of course, this optimal control problem does only make sense if it has at least one solution. This is established by the following Theorem:
\begin{thm}
	\label{NOC2}
	\hypertarget{HNOC2}
	The optimization problem \eqref{OP2} possesses a \textbf{globally optimal solution} $\u$, i.e., for all $u\in\UU$, $J(\u)\le J(u)$. If $\lambda_i>0$ for some index $i\in\{1,...,N\}$, it holds that 
	\begin{align*}
	\|\u_i\|_{L^2([0,T])} \le \frac 2 {\sqrt{\lambda_i}}\; \|\mathring f\|_{L^2(\RR^6)}.
	\end{align*}
\end{thm}

\begin{proof} $J$ is bounded from below since $J(u)\ge 0$ for all $u\in\UU$. Hence the infimum $M:={\inf}_{u\in\UU} J(u)$ exists and there also exists a minimizing sequence $(u_k)_{k\in\NN}\subset \UU$ such that $J(u_k) \to M$ if $k\to\infty$. As $\UU$ is weakly compact we obtain $u_k\rightharpoonup \u$ in $L^2([0,T];\RR^N)$ for some weak limit $\u\in\UU$ after extraction of a subsequence. This especially means that $[u_k]_i\rightharpoonup \u_i$ in $L^2([0,T])$ for every $i\in\{1,...,N\}$ and obviously also $B(u_k)\wto B(\u)\in\BB$ in $L^2(0,T;W^{2,\beta})$. Then, we can conclude from Proposition \ref{FSO}(c) that $f_{u_k} \wto f_\u$ in $W^{1,2}(0,T;L^2)$ up to a subsequence. By the fundamental theorem of calculus this directly implies that $f_{u_k}(T) \wto f_\u(T)$ in $L^2(\RR^6)$. Hence we can deduce from the weak lower semicontinuity of the $L^2$-norm that \vspace{-2mm}
	\begin{align*}
	J(\u) & = \frac 1 2 \|f_\u(T)-f_d\|_{L^2}^2 + \sum_{i=1}^N \frac {\lambda_i} 2 \|\u_i\|_{L^2}^2
	\le \underset{k\to\infty}{\lim\inf}\left[ \frac 1 2 \|f_{u_k}(T)-f_d\|_{L^2}^2 \right] +  \sum_{i=1}^N \frac{\lambda_i}{2}\, \underset{k\to\infty}{\lim\inf} \big\|[u_k]_i \big\|_{L^2}^2\\
	& \le \underset{k\to\infty}{\liminf}\left[ \frac 1 2 \|f_{u_k}(T)-f_d\|_{L^2}^2 +  \sum_{i=1}^N \frac {\lambda_i} 2 \big\| [u_k]_i\big\|_{L^2}^2 \right]
	= \underset{k\to\infty}{\lim}\; J(u_k) = M.
	\end{align*}
	By the definition of infimum this yields $J(\u) = M$. Now suppose that there exists some $i\in\{1,...,N\}$ such that $\lambda_i>0$ and $\|\u_i\|_{L^2([0,T])} > (2/\sqrt{\lambda_i}) \|\mathring f\|_{L^2(\RR^6)}$. Then
	\begin{align*}
	J(\u) &\ge \frac {\lambda_i} 2 \, \|\u_i\|_{L^2([0,T])}^2 > \frac 1 2 \, \big(2\,\|\mathring f\|_{L^2(\RR^6)}\big)^2 = \frac 1 2\; \big( \|f_0(T)\|_{L^2} + \|f_d\|_{L^2}\big)^2 
	\ge \frac 1 2\; \|f_0 (T)-f_d\|_{L^2}^2 = J(0)
	\end{align*}
	where $0$ denotes the null function $(0,...,0)^T \in\UU$. This, however, is a contradiction to the global optimality of $\u$ and thus the asserted inequality follows.
\end{proof}

Of course, this theorem does not provide uniqueness but only existence of a globally optimal solution. Since the control-to-state operator $u\mapsto f_u$ is nonlinear we cannot expect the cost functional $J$ to be convex. This means that the optimal control problem may also have several locally optimal solutions. In the following subsection, these locally optimal solutions will be characterized by necessary optimality conditions of first order.

\subsection{Necessary conditions for local optimality}
Since our set of admissible controls is a box-restricted subset of $L^2([0,T];\RR^N)$ this provides better possibilities to establish necessary optimality conditions compared to the model in \cite{knopf} where the magnetic field was the control itself. 
As the basic approach will be quite similar we will also have to discuss the costate equation. In this context we will need the constant $R_Z>0$ from \cite[Lem.\,6]{knopf}. It was defined in such a way that for any admissible field $B\in\BB$ the solution $Z_B=Z_B(s,t,x,v)$ of the characteristic system 
$$\dot x = v,\quad  \dot v= -\delx\psi_{f_B}(s,x) + v\times B(s,x)$$
with $Z_B(t,t,x,v)=(x,v)$ satisfies $\|Z(s,t,\cdot)\|_{L^\infty(\BR)} \le R_Z$ for all $s,t\in[0,T]$. Without loss of generality, we may assume that $R_Z\ge R$.
\begin{prop}
	\label{ADJS-u}
	\hypertarget{HADJS-u}
	Let $u\in \UU$ be arbitrary and let $f_u = \S(u)$ be its induced state that is given by the control-to-state operator. Suppose that $\chi\in C_c^2(\RR^6;[0,1])$ with $\chi=1$ on $B_{R_Z}(0)$ and 
	$\supp\chi\subset B_{2R_Z}(0)$. Then the \textbf{costate equation} 
	\begin{align}
	\label{COSTEQ2}
	\delt g + v\cdot\delx g  - \delx\psi_{f_u}\cdot\delv g + \big(v\times B(u)\big)\cdot\delv g = \Phi_{f_u,g}\ \chi, \quad g\big\vert_{t=T}=f_u(T)-f_d
	\end{align}
	where
	\begin{align*}
	\Phi_{\varphi,\gamma}(t,x):= - \iint \frac{x-y}{|x-y|^3}\cdot \delv\varphi(t,y,w)\, \gamma(t,y,w) \;\mathrm dy \mathrm dw, \quad t\in[0,T], x\in\RR^3
	\end{align*}
	has a unique strong solution $$g_u \in W^{1,2}(0,T;C_b)\cap C([0,T];C^1_b(\RR^6)) \cap L^\infty(0,T;H^2(\RR^6))$$ with $\supp g_u(t) \subset B_{R^*}(0)$, ${t\in[0,T]}$ for some constant $R^*>0$ depending only on $\mathring f,f_d,T,K$ and $\beta$. Note that $g_u\big\vert_{B_R(0)}$ does not depend on the choice of $\chi$. The operator 
	\begin{align*}
	\A:\UU\to C\big([0,T];L^2(\RR^6)\big),\; u\mapsto \A(u):=g_u
	\end{align*}
	is called the \textbf{control-to-costate operator}. It depends Lipschitz/Hölder-continuously on $u$ in such a way that there exists some constant $C\ge 0$ depending only on $\mathring f,f_d,T,K,\beta$ and $\|\chi\|_{C^1_b}$ such that
	\begin{align*}
	\|\A(u_1) - \A(u_2) \|_{C([0,T];C_b)}  &\le C\|u_1 - u_2\|_{L^2([0,T];\RR^N)},\\
	\|\A(u_1) - \A(u_2) \|_{W^{1,2}(0,T;C_b)} + \|\A(u_1) - \A(u_2) \|_{C([0,T];C^1_b)}  &\le C\|u_1 - u_2\|_{L^2([0,T];\RR^N)}^{\gamma}
	\end{align*}
	for all $u_1,u_2\in \UU$.\pskip
\end{prop}

\smallskip

\begin{proof}
	Since $u\in\UU$ and thus $B(u)\in\BB$, this result follows directly from \cite[Thm.\,5]{knopf} and the estimate
	$$\|B(u_1)-B(u_2)\|_{\WW} \le C\|u_1 - u_2\|_{L^2([0,T];\RR^N)},\quad u_1,u_2\in\DD $$
	that is a direct consequence of Lemma \ref{BOP}(b).
\end{proof}

Of course the costate equation \eqref{COSTEQ2} does not appear out of thin air. In the proof of Theorem \ref{MIN}, this equation will be derived by Lagrangian technique. Using the costate, various equivalent necessary conditions for local optimality can be established. In the following, the most important ones are presented:

\begin{thm}
	\label{MIN}
	\hypertarget{HMIN}
	Suppose that $\lambda_i >0$ for every $i\in\{1,...,N\}$ and let $\u\in\UU$ be any function. According to the definition of the control-to-state operator $f_\u$ denotes the unique strong solution of the state equation to the field $B(\u)\in\BB$. Moreover let $g_\u$ denote the unique strong solution of the costate equation \eqref{COSTEQ2}. We define the function $p(\u): [0,T] \to \RR^N$ by
	$p(\u)=(p_1(\u),...,p_N(\u))^T$ with
	$$p_i(\u)(t) := \int \big(v\times m_i(x)\big)\cdot\delv f_\u(t,x,v)\; g_\u(t,x,v)\dxv,\quad i=1,...,N\;.$$
	For every $\u\in\UU$, $p(\u)\in C([0,T];\RR^N)$.\pskip
	
	\pagebreak[1]
	\noindent The following items are equivalent:
	\begin{itemize}
		\item [\textnormal{\bfseries(NC1)}] $\u$ satisfies the \textbf{variational inequality}, i.e., for all $u=(u_1,...,u_N)\in\UU$,
		\begin{align*}
		\int\limits_0^T \big({\lambda_i}\u_i -p_i(\u)\big)\; (u_i-\u_i) \dt \ge 0,\quad i=1,...,N.
		\end{align*}
		\item [\textnormal{\bfseries(NC2)}] $\u$ is given implicitely by the \textbf{projection formula}, i.e., for almost all $t\in[0,T]$ and any $i\in\{1,...,N\}$,
		\begin{align*}
		\u_i(t) = \PP_{[a_i(t),b_i(t)]}\left(\frac 1 {{\lambda_i}} \, p_i(\u)(t)  \right)
		\end{align*}
		where $\PP_{[a,b]}$ denotes a projection of $\RR$ onto the interval $[a,b]$ that is given by $$\PP_{[a,b]}(\xi) = \min\big\{\max\{\xi,a\} , b\big\}, \quad \xi\in\RR.$$
	\end{itemize}
	Now suppose that $\u$ is a \textbf{locally optimal solution} of the optimization problem \eqref{OP2}, i.e., there exists $\delta>0$ such that $J(\u)\le J(u)$ for every $u\in\UU$ with ${\|\u-u\|_{L^2}<\delta}$. Then $\u$ satisfies the assertions \textnormal{(NC1)} and \textnormal{(NC2)}. This means that these items are (equivalent) necessary conditions for local optimality. \pagebreak[2]
\end{thm}

\begin{com}
	$\,$
	\begin{itemize}
		\item [\textnormal{(a)}] We can establish similar results if $\lambda_i=0$. The item (NC1) stays true in this case if we just replace $\lambda_i$ by zero. Instead of (NC2) we only have
		\begin{align*}\\[-7.5mm]
		\u_i(t) = a_i(t)\quad \text{if}\quad p_i(\u)(t) > 0 \quad\tand\quad	\u_i(t) =b_i(t)\quad\text{if}\quad p_i(\u)(t) < 0\\[-7mm]
		\end{align*}
		but $\u_i(t)$ is undefined if $p_i(\u)(t) = 0$.
		\item [\textnormal{(b)}] If $a_i$ and $b_i$ are continuous, so is $\u_i$ due to item (NC2). If this holds for all ${i\in\{1,...,N\}}$ we know that $\u$ is continuous and consequently $B(\u)\in C([0,T];C^{1,\gamma})$. In this case $f_\u$ and $g_\u$ are classical solutions of their respective systems.
	\end{itemize}
\end{com}

\bigskip

\begin{proof}
	The assertion $p_i(\u)\in C([0,T])$ is obvious since $f_{\u}, g_\u\in C([0,T];C^1_b)$ and $m_i\in C^{1,\gamma}(\RR^3;\RR^3)$ for all indices $i\in\{1,...,N\}$. First we will show that item (NC1) holds if $\u$ is a locally optimal solution. Therefore we apply the Lagrangian technique: 
	For $u\in\UU$ and $f,g\in H^1(]0,T[\times\RR^6)$ with $\supp f(t)\subset\BR$ for all $t\in[0,T]$ we define the Lagrangian\vspace{-2mm}
	\begin{align*}
	\Lag(f,u,g)&:= \frac 1 2 \|f(T)-f_d\|_{L^2} + \sum_{i=1}^N \frac {\lambda_i}{2} \|u_i\|_{L^2}^2 
	- \hspace{-3mm}\int\limits_{[0,T]\times \RR^6} \hspace{-3mm}\big( \delt f +v\cdot \delx f - \delx\psi_f\cdot\delv f + (v\times B(u))\cdot\delv f \big)\, g \dtxv\,.
	\end{align*}
	It is possible to replace $[0,T]\times \RR^6$ by $[0,T]\times \BR$ because of the support condition on $f$. As $\divergence_v (v\times B(u)) = 0$, integration by parts yields\vspace{-4mm}
	\begin{align*}
	\Lag(f,u,g)
	&= \frac 1 2 \|f(T)-f_d\|_{L^2} + \sum_{i=1}^N \frac {\lambda_i}{2} \|u_i\|_{L^2}^2 + \langle g(0),f(0) \rangle_{L^2} - \langle g(T),f(T) \rangle_{L^2}\\
	&\qquad + \hspace{-3mm}\int\limits_{[0,T]\times \BR} \hspace{-3mm} \big( \delt g +v\cdot \delx g - \delx\psi_f\cdot\delv g + (v\times B(u))\cdot\delv g \big)\; f \dtxv\;.
	\end{align*}
	The Lagrangian is partially Fréchet differentiable with
	\begin{align*}
	(\partial_f \Lag)(f,u,g)[h]
	&= \langle f(T)-f_d, h(T) \rangle_{L^2} - \langle g(T) , h(T) \rangle_{L^2} + \langle g(0), h(0) \rangle_{L^2}\\[2mm]
	&\;\; + \hspace{-3mm}\int\limits_{[0,T]\times \BR} \hspace{-4mm} \big( \delt g +v\cdot \delx g - \delx\psi_f\cdot\delv g + (v\stimes B(u))\cdot\delv g \big)\, h \dtxv
	- \hspace{-3mm}\int\limits_{[0,T]\times \BR} \hspace{-3mm} \Phi_{f,g}(t,x)\; h \dtxv
	\end{align*}
	for any $h\in H^1(]0,T[\times\RR^6)$ and\vspace{-2mm}
	\begin{align*}
	&(\partial_u \Lag)(f,u,g)[h] = \sum_{i=1}^N \lambda_i \langle u_i,h_i \rangle_{L^2} \;-\hspace{-3mm} \int\limits_{[0,T]\times \RR^6}\hspace{-3mm} (v\times B(h))\cdot\delv f\; g \dtxv
	%
	%
	\end{align*}
	for any $h\in \DD$. Obviously $J(\u)=\Lag(f_\u,\u,g)$ and hence
	\begin{align*}
	J'(\u)[h] = (\partial_f \Lag)(f_\u,\u,g)\big[f'_{\u}[h]\big] + (\partial_u \Lag)(f_\u,\u,g)[h] 
	\end{align*}
	for all $h\in\DD$. Now, if $\u$ is a local minimizer of $J$ then $J'(\u)[h]$ is nonnegative for all directions $h\in\DD$ with $u+h\in\UU$. Thus inserting $g=g_{\u}$ yields $(\partial_f \Lag)(f_\u,\u,g)\big[f'_{\u}[h]\big]=0$ and thus
	\begin{align}
	\label{FDJ}
	0 \le J'(\u)h = (\partial_u \Lag)(f_\u,\u,g_\u)[h] = \sum_{i=1}^N\; \int\limits_0^T \big( \lambda_i\, \u_i(t) - p_i(\u)(t) \big) h_i(t) \dt 
	\end{align}
	for all $h\in \DD$ with $u+h\in\UU$. For any fixed $i\in\{1,...,N\}$ we can choose $h_j = 0$ if $j\neq i$ while $h_i$ is still arbitrary. This finally implies that
	\begin{align*}
	\int\limits_0^T ({\lambda_i} \u_i - p_i(\u)) \; h_i \dt \ge 0, \quad  i=1,...,N
	\end{align*}
	for all $h\in \DD$ with $u+h\in\UU$. For any $u\in\UU$ we can now choose $h:=u-\u\in\DD$ and hence we can conclude that for all $u\in\UU$,
	\begin{align*}
	\int\limits_0^T ({\lambda_i} \u_i - p_i(\u)) \; (u_i-\u_i) \;\mathrm dt\ge 0,\quad i=1,...,N
	\end{align*}
	that is (NC1). The equivalence of (NC1) and (NC2) is a standard result (see, e.g., \cite[pp.\,67-71]{troeltzsch}).	
\end{proof}

\bigskip

If $\u\in\UU $ is a locally optimal control we can also show that the triple $(f_\u,g_\u,\u)$ satisfies a certain system of partial differential equation that will be referred to as the \textbf{optimality system} of the optimization problem. A strong solution of the optimality system is defined as follows: \smallskip\pagebreak[3]

\begin{defn}
	Suppose that $\lambda_i>0$ for all $i\in\{1,...,N\}$. The triple $(f,g,u)$ is called a strong solution of the \textbf{optimality system} iff the following conditions hold:
	\begin{itemize}
		\item[\textnormal{(i)}] $f,g\in W^{1,2}(0,T;C_b)\cap C([0,T];C^1_b)$ and $u\in L^2([0,T];\RR^N)$.\vspace{2.5mm}
		\item[\textnormal{(ii)}] For any $t\in[0,T]$, $\supp f(t) \subset B_R(0)$ and $\supp g(t) \subset B_{R^*}(0)$ where $R,R^*>0$ are the constants from {Theorem \ref{GWS}} and {Theorem \ref{ADJS-u}}.\vspace{2.5mm}
		\item[\textnormal{(iii)}] $f$, $g$ and $u$ satisfy the following system of equations almost everywhere:
		\begin{align}
		\label{OS-SEQ-u}
		\hspace{8pt}
		\begin{cases}
		\delt f + v\cdot\delx f  - \delx\psi_{f}\cdot\delv f + (v\times B(u))\cdot\delv f = 0, & f\big\vert_{t=0} = \mathring f,\\[3mm]
		\delt g + v\cdot\delx g  - \delx\psi_{f}\cdot\delv g + (v\times B(u))\cdot\delv g = \Phi_{f,g}\chi, & g\big\vert_{t=T} = f(T)-f_d,\\[3mm]
		u=(u_1,...,u_N)^T \text{ with } u_i = \PP_{[a_i,b_i]} \Big( \frac 1 {\lambda_i} \textstyle\int (v\times m_i)\cdot \delv f\; g \dxv \Big).
		\end{cases}
		\end{align}
	\end{itemize}
\end{defn}

\bigskip
\pagebreak[2]
We obtain the following condition for local optimality:\\[-6mm]

\begin{cor}
	\label{OS2}
	\hypertarget{HOS2}
	Suppose that $\u\in\UU$ is a locally optimal solution of the optimization problem \eqref{OP2}. Then the following holds:
	\begin{itemize}
		\item[\textnormal{\bfseries(NC3)}] $(f_\u,g_\u,\u)$ is a strong solution of the \textbf{optimality system} \eqref{OS-SEQ-u}.
	\end{itemize}
	Moreover, condition \textnormal{(NC3)} is equivalent to the necessary optimality conditions \textnormal{(NC1)} and \textnormal{(NC2)} .
\end{cor}

\begin{proof}
	It is obvious that condition (NC3) is equivalent to (NC2). Hence it is a necessary condition for local optimality that is equivalent to the conditions (NC1) and (NC2) according to Theorem \ref{MIN}.
\end{proof}

\smallskip

Note that for any locally optimal control $\u$ the costate $g_\u$ can be considered as a Lagrangian multiplier with respect to the side condition on $f_\u$ that is the state equation. However, we can also find Lagrangian multipliers with respect to the control restrictions. This makes it possible to write the optimal control problem as a \textbf{Karush-Kuhn-Tucker system}.

\begin{cor}
	Suppose that $\u\in\UU$ is a locally optimal solution of the optimization problem \eqref{OP2}. Then the following holds:
	\begin{itemize}
		\item[\textnormal{\bfseries(NC4)}] $\u$ satisfies the \textbf{Karush-Kuhn-Tucker conditions}, i.e., there exist Lagrangian multipliers $\mu^a,\mu^b$ in \linebreak $L^2([0,T];\RR^N)$ such that:\vspace{1mm}
			\item[\textnormal{(i)}] Primal feasibility: For all $i\in\{1,...,N\}$, $$a_i \le  \u_i\le b_i \quad\text{almost everywhere on}\; [0,T].$$
			\item[\textnormal{(ii)}] Dual feasibility: For all $i\in\{1,...,N\}$, $$ \mu^a_i , \mu^b_i \ge 0 \quad\text{almost everywhere on}\; [0,T].$$
			\item[\textnormal{(iii)}] Complementary slackness: For all $i\in\{1,...,N\}$, $$\mu^a_i (\u_i - a_i) = 0,\quad \mu^b_i (\u_i - b_i) = 0\quad\text{almost everywhere on}\; [0,T].$$ 
			\item[\textnormal{(iv)}] Stationarity: For all $i\in\{1,...,N\}$, $$ \lambda_i\u_i - p_i(\u) - \mu^a_i + \mu^b_i = 0 \quad\text{almost everywhere on}\; [0,T].$$ 
	\end{itemize}
	Moreover, condition \textnormal{(NC4)} is equivalent to the necessary optimality conditions \textnormal{(NC1)} and \textnormal{(NC2)}.
\end{cor}

Since this is a standard result, we do not present a detailed proof. For a similar proof see \cite[pp.\,71-73]{troeltzsch}.

\subsection{Uniqueness of the optimal solution on small time intervals} 

We can show that the solution of the optimality system is unique if the final time $T$ is small compared to the regularization parameters $\lambda_i$. From this, we can conclude that the optimal control is unique on sufficiently small time intervals. This approach to obtain uniqueness has already been used in literature in similar situations (see, e.g., \cite{bradley-lenhart,fister-panetta}). We also want to mention the paper \cite{hinze} where a uniqueness result for the optimal control was established under the condition that the costate is bounded in a certain $L^q$-norm.

\begin{thm}
	\label{UOS2}
	\hypertarget{HUOS2}
	Let $\lambda>0$ be defined by $\lambda := \min \{\lambda_1,...,\lambda_N\}$. Suppose that $\lambda\in ]0,1]$ and let us assume that there exists a strong solution $(f,g,u)$ of the optimality system $\eqref{OS-SEQ-u}$. Then this solution is unique if the quotient $\tfrac T \lambda$ is sufficiently small.
\end{thm}

\begin{proof}
	Suppose that the triple $(\tf,\tg,\tu)$ is another strong solution. Let $C\ge 0$ denote some generic constant that may depend on $T$, $N$, $M$, $a$, $b$, $\mathring f$ and $f_d$. Since $\supp f(t)$, $\supp \tf(t) \subset B_R(0)$ we have for almost all $t\in[0,T]$ \vspace{-1mm}
	\begin{align}
	\label{UNIQU}
	|u_i(t)-\tu_i(t)| &= \left|  \PP_{[a_i(t),b_i(t)]}\left(\frac 1 {{\lambda_i}} p_i(u)(t) \right) - \PP_{[a_i(t),b_i(t)]}\left(\frac 1 {{\lambda_i}} p_i(\tilde u)(t) \right) \right|\notag\\[2mm]
	&\le \frac 1 {{\lambda_i}}\, \big| p_i(u)(t) - p_i(\tilde u)(t) \big|\notag\\
	&\le \frac C {{\lambda}}\; \|g_u\|_{C([0,T];C_b)}\|f_u(t) - f_\tu (t)\|_\infty + \frac C {{\lambda}}\; \|f_\tu\|_{C([0,T];C^1_b)}\|g_u(t) - g_\tu (t)\|_\infty
	\end{align}
	and hence\vspace{-2mm}
	\begin{align*}
	\|B(u)(t)-B(\tu)(t)\|_\infty \le \frac C {{\lambda}}\; \|f_u(t) - f_\tu (t)\|_\infty + \frac C {{\lambda}}\; \|g_u(t) - g_\tu (t)\|_\infty\;
	\end{align*}
	where now $C=C(T)$ also depends on $\|g_u\|_{C([0,T];C_b)}$ and $\|f_\tu\|_{C([0,T];C^1_b)}$ and is monotonically increasing in $T$. Note that, for any fixed $T>0$, the $C([0,T];C^1_b)$-norm of $f_\tu$ is bounded uniformly for any $\tu\in\mathbb U$ by virtue of Proposition~\ref{GWS}(a) so that the new constant $C(T)$ can be assumed to be independent of $(\tf,\tg,\tu)$.\pskip
	
	Let now $Z_u=Z_u(s,t,z)$ denote the solution of the characteristic system
	\begin{align*}
	\dot x = v, \quad \dot v = -\delx\psi_f(s,x) + v\times B(u)(s,x)
	\end{align*}
	to the initial condition $Z_u(t,t,z)=z$ and let $Z_\tu$ be defined analogously. Now, let $s,t\in [0,T]$ (without loss of generality, $s\le t$) and $z\in\RR^6$ be arbitrary. Then
	\begin{align*}
	|Z_u(s,t,z) - Z_\tu(s,t,z)| 
	&\le \int\limits_s^t |Z(\tau,t,z)-Z_\tu(\tau,t,z)| + \|\delx\psi_{f-\tf}(\tau)\|_\infty + R\,\|B(u)(\tau)-B(\tu)(\tau)\|_\infty \dtau \\
	&\le \int\limits_s^t |Z(\tau,t,z)-Z_\tu(\tau,t,z)| + \tfrac C \lambda \|f(\tau)-\tf(\tau)\|_\infty +  \tfrac C \lambda \|g(\tau)-\tg(\tau)\|_\infty\dtau.
	\end{align*}
	and thus by Gronwall's lemma,
	\begin{align*}
	&|Z_u(s,t,z) - Z_\tu(s,t,z)| \le \tfrac C \lambda \int\limits_s^t  \|f(\tau)-\tf(\tau)\|_\infty +  \|g(\tau)-\tg(\tau)\|_\infty\dtau.
	\end{align*}
	Analogously to the proof of \cite[Thm.\,8]{knopf}, we can conclude that
	\begin{align}
	\|f - \tf\|_{C([0,T];C_b)} &\le C \tfrac T \lambda \exp\left( C \tfrac T \lambda \right) \|g - \tg\|_{C([0,T];C_b)},\\
	\label{UNIQG}
	\|g - \tg\|_{C([0,T];C_b)} &\le C \tfrac T \lambda \exp\left( C \tfrac T \lambda \right) \|f - \tf\|_{C([0,T];C_b)}
	\end{align}
	and consequently 
	\begin{align*}
	\|f - \tf\|_{C([0,T];C_b)} \le \big(C \tfrac T \lambda \big)^2 \exp\left( C \tfrac T \lambda \right) \|f - \tf\|_{C([0,T];C_b)}.
	\end{align*}
	If now $\tfrac T \lambda$ is sufficiently small, we have $\big(C \tfrac T \lambda \big)^2 \exp\left( C \tfrac T \lambda \right) < 1$ and it follows that $	\|f - \tf\|_{C([0,T];C_b)}=0$. This means that $f=\tf$ and then we obtain $g=\tg$ by \eqref{UNIQG} and finally $u=\tu$ by \eqref{UNIQU}.
\end{proof}

\bigskip

Finally, we can easily deduce uniqueness of the globally optimal solution if $\frac T \lambda$ is small:

\begin{cor}
	Let $\lambda$ denote the parameter from Theorem~\ref{UOS2}. Again, we assume that $\lambda\in ]0,1]$ and that $\tfrac T \lambda$ is sufficiently small. Then the following holds:
	\begin{itemize}
		\itema The optimal control problem \eqref{OP2} has a unique globally optimal solution.
		\itemb $\u$ is the unique globally optimal solution of \eqref{OP2} if and only if $\u$ satisfies the optimality conditions \textnormal{(NC1)-(NC4)}. In this case, \textnormal{(NC1)-(NC4)} are necessary and sufficient conditions for global optimality.
	\end{itemize}
\end{cor}

\begin{proof} 
	First note that there exists at least one globally optimal solution $\u$ according to Theorem \ref{NOC2}. Then, of course, $\u$ is also locally optimal and satisfies condition (NC3). If now $\tfrac T \lambda$ is sufficiently small, we can conclude from Theorem \ref{UOS2} that $\u$ is uniquely determined by the optimality system. Hence there is exactly one globally optimal solution of the optimal control problem and that is $\u$. This proves (a).\pskip
	
	If $\u$ is the unique globally optimal solution it is obviously also locally optimal and satisfies the equivalent conditions for local optimality (NC1)-(NC4). To prove the reverse implication, we will now assume that there exists some control $u\in\UU$ that satisfies the conditions (NC1)-(NC4). It remains to show that $u$ is then the unique globally optimal solution. Recall that according to item (a) there is at least one unique globally optimal solution $\u$ of the optimal control problem. As both $u$ and $\u$ satisfy the necessary condition (NC3) it must hold that $u=\u$ if $\tfrac T \lambda$ is small enough. This proves the equivalence assertion of item (b). Therefore, in this case, (NC1)-(NC4) are necessary and sufficient conditions for global optimality as they are satisfied only by the unique globally optimal solution.
\end{proof}

\subsection{A sufficient condition for local optimality}

In the previous section, we have showed that the necessary conditions (NC1)-(NC4) are also sufficient conditions if $\frac T \lambda$ is small enough. We will now establish a sufficient condition for (even strict) local optimality without direct restrictions on $T$ and $\lambda$. Therefore we will need the cost functional $J$ to be twice continuously Fréchet differentiable. Unfortunately, we cannot prove that the control-to-state operator $\S$ is twice Fréchet differentiable, which is due to the following: Linearizing \eqref{FDEQ} once more, we see that the Vlasov equation which $\F''(u)(h,\th)$ should satisfy contains terms $\delv\big(\F'(u)[h]\big)$ and $\delv\big(\F'(u)[\th]\big)$ as parts of a source term. Since $\F'(u)[h]$ and $\F'(u)[\th]$ are only of class $H^1$, the above derivatives are only of class $L^2$ which is not enough to show solvability of that new twice linearized Vlasov equation (cf. \cite{knopf}). Also, standard approaches via the implicit function theorem fail due to the loss of regularity of solutions to the linearized equation \eqref{FDEQ}.\pskip

However, it will do fine to have first-order Fréchet differentiability of the field-costate operator $\A$. This is established by the following Lemma:

\begin{prop}
	\label{FDCCSO2}
	\hypertarget{HFDCCSO2}
	The field-costate operator $\A$ is Fréchet differentiable on $\UU$. For any control $u\in\UU$ and any direction $h\in\DD$, the Fréchet derivative $\A'(u)[h]=g'_u[h]$ is the unique strong solution of the initial value problem
	\begin{align}
	\label{FDEQ2}
	\begin{cases}
	& \hspace{-4mm}\delt g + v\scdot\delx g - \delx\psi_{f'_u[h]}\scdot\delv g_u - \delx\psi_{f_u}\scdot\delv g 
	+ \big(v\stimes B(u)\big)\scdot\delv g + (v\stimes B(h))\scdot\delv g_B 
	= \Phi_{f_u,g} \chi - \Phi_{g_u,f'_u[h]}\chi, \\[2mm]
	& \hspace{-4mm}g\big\vert_{t=T} = 0.
	\end{cases}
	\end{align}
	where $\Phi$ is the operator from Proposition \ref{ADJS-u}. This means that $\A'(u)[h]$ lies in  $(L^\infty \cap H^1)(]0,T[\times\RR) \subset C([0,T];L^2)$, satisfies \eqref{FDEQ2} almost everywhere and there exists some radius $\varrho>0$ depending only on $T,K,\mathring f$ and $\beta$ such that $\supp g'_u[h](t) \subset B_\varrho(0)$ for all $t\in[0,T]$.\pskip
	
	\noindent Moreover, the Fréchet derivative $\A'(u)[h]$ depends Hölder-continuously on $u\in\UU$ in the following sense: There exists some constant $C>0$ depending only on $\mathring f$, $T$, $K$ and $\beta$ such that for all $u,\tu\in\UU$ and $h\in\DD$,
	\begin{align} 
	\label{CFDEQ2}
	\underset{\|h\|_{L^2}\le 1}{\sup}\; \|\A'(u)[h] - \A'(\tu)[h]\|_{L^2(0,T;L^2)} \le C\, \|u-\tu\|_\DD^\gamma. 
	\end{align}
\end{prop}

\begin{proof}
	We already know from \cite[Lem. 10]{knopf} (which obviously holds true if the space $\VV=\WW\cap L^2(0,T;H^1)$ is replaced by $\WW$) that the operator 
	$$g.:\BB\to C\big([0,T];L^2(\RR^6)\big),\quad B\mapsto g_B$$ 
	is Fréchet differentiable. The Fréchet derivative $g'_B[H]$ is the unique strong solution of\vspace{-2mm}
	\begin{align}
	\label{GDEQ}
	\begin{cases}
	& \hspace{-3mm}\delt g + v\cdot\delx g - \delx\psi_{f'_B[H]}\cdot\delv g_u - \delx\psi_{f_B}\cdot\delv g 
	+ \big(v\stimes B\big)\cdot\delv g + (v\stimes H)\cdot\delv g_B 
	= \Phi_{f_B,g} \chi - \Phi_{g_B,f'_B[H]}\chi, \\[2mm]
	& \hspace{-3mm}g\big\vert_{t=T} = 0.
	\end{cases}
	\end{align}
	This means that $g'_B[H]$ lies in $L^\infty \cap H^1(]0,T[\times\RR) \subset C([0,T];L^2)$ satisfies \eqref{GDEQ} almost everywhere and there exists $\varrho>0$ depending only on $T,K,\mathring f$ and $\beta$ such that $\supp g'_B[H](t) \subset B_\varrho(0)$ for all $t\in[0,T]$ and\vspace{-1mm}
	\begin{align}
	\label{CFG}
	\underset{\|H\|_\WW \le 1}{\sup}\;\|g_A'[H] - g_B'[H]\|_{L^2(0,T;L^2)} \le C \;\|A-B\|_{\WW}^{\gamma}, \quad A,B \in \BB.
	\end{align}
	As the control-to-field operator $u\mapsto B(u)$ is Fréchet differentiable according to Proposition \ref{BOP}(c), the chain rule implies that the control-costate operator  $u\mapsto g_u = g_{B(u)}$ is also Fréchet differentiable and the Fréchet derivative is given by $g'_u[h] = g'_{B(u)}[B(h)]$. Hence, $g'_u[h]$ is the unique strong solution of \eqref{FDEQ2} with $\supp g'_u[h](t) \subset B_\varrho(0)$ for all $t\in[0,T]$ and condition \eqref{CFDEQ2} holds.
\end{proof}

\bigskip

From this result we can conclude that the cost functional is twice continuously Fréchet differentiable. 

\pagebreak[3]

\begin{prop}
	\label{TFDJ2}
	\hypertarget{HTFDJ2}
	$\;$
	\begin{itemize}
		\itema The cost functional $J$ of the optimization problem \eqref{OP2} is twice Fréchet differentiable on $\UU$. The Fréchet derivative of second order at the point $u\in\UU$ can be described as a bilinear operator 
		\begin{align*}
		&J''(u):\DD^2 \to \RR, \\
		&J''(u)[h,\th] 
		= \sum_{i=1}^N\left\{ 
		\lambda_i\, \langle h_i, \th_i \rangle_{L^2([0,T])} - \hspace{-3mm}\int\limits_{[0,T]\times\RR^6}\hspace{-3mm} (v\stimes m_i)\scdot \big( \delv f_u\, g_u'[\th] - \delv g_u\, f_u'[\th] \big)\, h_i \dtxv
		\right\}
		\end{align*}
		for all $h,\th \in \DD$. 
		\itemb There exists some constant $C>0$ depending only on $\mathring f$, $f_d$, $N$, $M$, $a$, $b$, $T$ and $\beta$ such that for all $u,\tu\in\UU$,
		\begin{align}
		\label{CFDJ}
		\|J''(u) - J''(\tu) \| \le C\, \|u-\tu\|_{L^2([0,T];\RR^N)}^\gamma
		\end{align}
		where
		\begin{align*}
		\| J''(u) \| = \sup\Big\{ \big|J''(u)[h_1,h_2] \big| \,\Big\vert\, \|h_1\|_{L^2([0,T];\RR^N)} = 1 ,\, \|h_2\|_{L^2([0,T];\RR^N)} = 1\Big\}
		\end{align*}
		denotes the operator norm. This means that $J$ is twice continuously differentiable.
	\end{itemize}
\end{prop}

\begin{proof}
	Let $C>0$ denote some generic constant depending only on $\mathring f$, $f_d$, $N$, $M$, $a$, $b$, $T$ and $\beta$. In \eqref{FDJ} we have already proved that
	\begin{align*}
	J'(u)[h] = \sum_{i=1}^N\; \int\limits_0^T \left[ \lambda_i\, u_i(t) - \left( \int (v\times m_i)\cdot\delv f_u\, g_u \dxv \right) \right] h_i(t) \dt
	\end{align*}
	for any $u\in\UU$ and all $h\in\DD$. We will now prove that $J'(u)$ is once more Fréchet differentiable with respect to $u$. For $\eps>0$ we define the sets
	\begin{align*}
	&\UU_i^\eps:= \big\{ u\in\DD \;\big\vert\; \|u\|_{L^2([0,T])} < \|a_i\|_{L^2([0,T])} + \|b_i\|_{L^2([0,T])} + \eps \big\},\quad i=1,...,N,\\
	&\UU^\eps:=\UU_1^\eps \times \;...\; \times \UU_N^\eps\;.
	\end{align*}
	Then, $\UU^\eps$ is open with $\UU\subsetneq \UU^\eps$ and, similarly to the proof of Proposition \ref{BOP}(b), we can conclude that $B(\UU^\eps) \subset \IBB$ if $\eps$ is sufficiently small. Let now $u$ be any function in $\UU^\eps$ and let $\th\in\DD$ be arbitrary with $u+\th\in\UU^\eps$. By Taylor expansion we obtain the decompositions
	\begin{align}
	\label{TEFG}
	f_{u+\th} - f_u = f'_u[\th] + f^R_u[\th] \tand g_{u+\th} - g_u = g'_u[\th] + g^R_u[\th]
	\end{align}
	where $f^R_u[\th], g^R_u[\th] \in C([0,T];L^2)$. It follows from \cite[Thm.\,3,\,Lem.\,10]{knopf} and the boundedness of $B(\cdot)$ that
	\begin{align}
	\label{TEFG2}
	\begin{aligned}
	\|f^R_u[\th]\|_{C([0,T];L^2)} \le C\|B(\th)\|_{\WW}^{1+\gamma} \le C\|\th\|_{L^2([0,T];\RR^N)}^{1+\gamma},\quad
	\|g^R_u[\th]\|_{C([0,T];L^2)} \le C\|B(\th)\|_{\WW}^{1+\gamma} \le C\|\th\|_{L^2([0,T];\RR^N)}^{1+\gamma}
	\end{aligned}
	\end{align}
	where $\gamma=1-\nicefrac 3 \beta>0$ denotes the constant from Proposition \ref{SAC}. Now, by integration by parts,
	\begin{align*}
	&J'(u+\th)[h] - J'(u)[h] \\
	&\quad=  \sum_{i=1}^N \Bigg\{ \lambda_i \langle h_i, \th_i \rangle_{L^2([0,T])} \;- \hspace{-3mm}\int\limits_{[0,T]\times\RR^6} \hspace{-3mm}(v\stimes m_i)\cdot (\delv f_{u+\th}- \delv f_u) (g_{u+\th}-g_u) \, h_i\dtxv \\
	&\qquad\qquad - \int\limits_{[0,T]\times\RR^6} \hspace{-3mm}(v\stimes m_i)\cdot \big[\delv f_u (g_{u+\th}-g_u) - \delv g_u (f_{u+\th}-f_u)\big] h_i\dtxv \Bigg\}.
	\end{align*}
	Inserting \eqref{TEFG} then yields
	\begin{align}
	\label{EQ:JDIFF}
	&J'(u+\th)[h] - J'(u)[h] \notag\\
	&\quad= \sum_{i=1}^N\hspace{-2pt}\Bigg\{ 
	\lambda_i \langle h_i, \th_i \rangle_{L^2([0,T])} - \hspace{-4mm}\int\limits_{[0,T]\times\RR^6}\hspace{-4mm} (v\stimes m_i)\scdot \big( \delv f_u\, g_u'[\th] - \delv g_u\, f_u'[\th] \big) h_i \dtxv
	\Bigg\} + \mathcal R(h,\th)
	\end{align}
	where
	\begin{align*}
	\mathcal R(h,\th) &:= - \sum_{i=1}^N \Bigg\{ \, \int\limits_{[0,T]\times\RR^6} \hspace{-3mm}(v\stimes m_i)\cdot (\delv f_u\, g^R_u[\th] - \delv g_u\, f^R_u[\th])\, h_i\dtxv\\
	&\qquad\qquad\qquad + \hspace{-4mm}\int\limits_{[0,T]\times\RR^6} \hspace{-3mm}(v\stimes m_i)\cdot (\delv f_{u+\th}- \delv f_u) (g_{u+\th}-g_u) h_i\dtxv \Bigg\}
	\end{align*}
	The remainder $\mathcal R(h,\th)$ can be bounded by 
	\begin{align*}
	|\mathcal R(h,\th)| \le C\, \|m_i\|_{L^\infty} \, \|h\|_{L^2([0,T];\RR^N)} &\Big( \|\delv f_u\|_{L^\infty}\, \|g^R_u[\th]\|_{C([0,T];L^2)} 
	+ \|\delv g_u\|_{L^\infty}\, \|f^R_u[\th]\|_{C([0,T];L^2)}\\
	&\qquad + \|\delv f_{u+\th} - \delv f_u\|_{L^\infty} \|g_{u+\th} - g_u\|_{L^2} \Big).
	\end{align*}
	From Proposition 3(a), Proposition \ref{ADJS-u} and \eqref{TEFG2} we deduce that $\mathcal R(\cdot,\th)$ can be bounded in the operator norm by
	\begin{align*}
		\|\mathcal R(\cdot,\th)\| = 
		\sup \big\{ |\mathcal R(h,\th)| : \|h\|_{L^2([0,T];\RR^N)} \le 1 \big\} \le C\|\th\|_{L^2([0,T];\RR^N)}^{1+\gamma} 
	\end{align*}
	Hence, we can conclude from equation \eqref{EQ:JDIFF} that $J$ is twice Fréchet differentiable on $\UU^\eps$ (and thus especially on $\UU$) where the second-order Fréchet derivative is given by
	\begin{align*}
	J''(u)[h,\th]
	= \sum_{i=1}^N\hspace{-2pt}\Bigg\{ 
	\lambda_i \langle h_i, \th_i \rangle_{L^2([0,T])} - \hspace{-4mm}\int\limits_{[0,T]\times\RR^6}\hspace{-4mm} (v\stimes m_i)\scdot \big( \delv f_u\, g_u'[\th] - \delv g_u\, f_u'[\th] \big) h_i \dtxv
	\Bigg\}
	\end{align*}
	since
	\begin{align*}
		\frac{ \|J'(u+\th)[\cdot] - J'(u)[\cdot] - J''(u)[\cdot,\th]\|}{ \|\th\|_{L^2([0,T];\RR^N)} }
		= \frac{ \|\mathcal R(\cdot,\th)\| }{ \|\th\|_{L^2([0,T];\RR^N)} }
		\le \|\th\|_{L^2([0,T];\RR^N)}^\gamma \to 0,
	\end{align*}
	as $\|\th\|_{L^2([0,T];\RR^N)}\to 0$ where $\|\cdot\|$ stands for the operator norm. We point out that the operator $\th\mapsto J''(u)[h,\th]$ is linear and bounded. For any $u,\tu \in \UU$ and $h,\th\in\DD$, 
	\begin{align}
	\label{CFDJ-EST}
	\begin{aligned}
	&|J''(u)[h,\th] - J''(\tu)[h,\th]| \\
	&\quad\le C \, \|h\|_{L^2([0,T];\RR^N)} \Big( \|\delv f_u - \delv f_\tu \|_{L^\infty} \|g'_\tu[\th]\|_{L^2} 
	+ \|\delv g_u - \delv g_\tu \|_{L^\infty} \|f'_u[\th]\|_{L^2} \\
	&\hspace{120pt} + \|f'_u[\th] - f'_\tu[\th] \|_{L^2} \|\delv g_\tu\|_{L^\infty} 
	+ \|g'_u[\th] - g'_\tu[\th] \|_{L^2} \|\delv f_u\|_{L^\infty}\Big).
	\end{aligned}
	\end{align} 
	Note that for $ \|\th\|_\DD \le 1 $ the terms $\|g'_\tu[\th]\|_{L^2}$, $\|f'_u[\th]\|_{L^2}$, $\|\delv g_\tu\|_{L^\infty}$ and $\|\delv f_u\|_{L^\infty}$ are all bounded from above by some constant depending only on $\mathring f, f_d, a,b,M$ and $T$. This is due to \cite[Cor.\,2]{knopf} in combination with Proposition~\ref{BOP}(b). Also recall that
	\begin{align*}
	\left.
	\begin{aligned}
	&\|\delv f_u - \delv f_\tu \|_{L^\infty} \\
	&\|\delv g_u - \delv g_\tu \|_{L^\infty} \\
	&\|f'_u[\th] - f'_\tu[\th] \|_{L^2} \\
	&\|g'_u[\th] - g'_\tu[\th] \|_{L^2}
	\end{aligned}
	\right\} \; \le \; C\, \|u-\tu\|_\DD^\gamma
	\end{align*}
	if $ \|\th\|_\DD \le 1 $ according to the Propositions \ref{FSO}(a), \ref{ADJS-u}, \ref{BOP}(d) and \ref{FDCCSO2}. Hence, if both $ \|h\|_\DD \le 1 $ and $ \|\th\|_\DD \le 1 $, \eqref{CFDJ-EST} implies that
	\begin{align*}
	|J''(u)[h,\th] - J''(\tu)[h,\th]| \le C\, \|u-\tu\|_\DD^\gamma.
	\end{align*} 
	This proves \eqref{CFDJ} and completes the proof of Proposition \ref{TFDJ2}.
\end{proof}

\begin{com}
	Since $J$ is twice continuously Fréchet differentiable, Schwarz's theorem yields that $J''(u)$ is symmetric, i.e. $J''(u)[h,\th]=J''(u)[\th,h]$ for all $u\in\UU$ and $h,\th\in\DD$.
\end{com}

Now that we know that the cost functional is twice continuously differentiable, we can easily establish a sufficient condition for strict local optimality: Let $\u\in\UU$ satisfy the variational inequality (NC1) and suppose that $J''(\u)$ is positive definite, i.e., there exists $c>0$ such that
\begin{align*}
J''(\u)[h,h] \ge c\, \|h\|_\DD^2, \quad h\in\DD.
\end{align*}
Then $\u$ is a strict local minimizer of $J$ on the set $\UU$. However this condition is far too restrictive and would resemble cracking a nut with a sledgehammer. (The authors' attention was drawn to this fact by L. Blank \cite{blank}). Therefore, we will present a much weaker sufficient condition for strict local optimality where only directions $h$ from a certain critical cone $\CC$ have to be taken into account. For general semilinear (elliptic or parabolic) control problems this method was introduced by E. Casas, J.\,C. de los Reyes and F.~Tröltzsch \cite{casas-reyes-troeltzsch}. We will proceed similarly and define the cone of critical directions as follows:

\begin{defn}
	We define the set
	\begin{align}
	\CC:=\left\{ h\in L^2([0,T];\RR^N) \bigvert h\; \text{satisfies condition}\; \eqref{COND:CONE} \right\}
	\end{align}
	where condition \eqref{COND:CONE} reads as follows:
	\begin{align}
	\label{COND:CONE}
	\text{For all}\; i\in\{1,...,N\} \;\text{and almost all}\; t\in[0,T]: \quad h_i(t) = \begin{cases} \ge 0, & \bar u_i(t) = a_i(t) \\ \le 0, & \bar u_i(t) = b_i(t) \\ = 0, & \lambda_i \u_i(t) - p_i(\u)(t)\neq 0 \end{cases}.
	\end{align}
	This set $\CC$ is called the \textbf{cone of critical directions}.
\end{defn}

\pagebreak[2]

Now, we can establish the following sufficient condition for strict local optimality:

\begin{thm}
	Suppose that $\u\in\UU$ and let $f_\u$ and $g_\u$ be its induced state and costate. 
	We assume that the variational inequality \textnormal{(NC1)} holds for all $u\in\UU$, i.e.,
	\begin{align*}
	&\int\limits_0^T \big( \lambda_i \u_i - p_i(\u) \big) (u_i-\u_i) \dt 
	\ge 0,\quad i=1,...,N,\; u\in\UU,
	\end{align*}
	and that for all critical directions $h\in \CC\setminus\{0\}$,
	\begin{align}
	\label{CIQ2}
	&\sum_{i=1}^N \left\{ \lambda_i \; \|h_i\|_{L^2([0,T])}^2
	- \hspace{-3mm}\int\limits_{[0,T]\times\RR^6}\hspace{-3mm} (v\stimes m_i)\cdot \Big( \delv f_\u\, g_\u'[h] - \delv g_\u\; f_\u'[h] \Big) \, h_i \dtxv \right\}
	= J''(\u)[h,h] > 0.
	\end{align}
	In this case $J$ satisfies a quadratic growth condition: There exist $\delta>0$ and $\eps>0$ such that for all $u\in\UU$ with \linebreak ${\|u-\u\|_{L^2([0,T];\RR^N)}<\delta}$,
	\begin{align}
	\label{QGC}
	J(u) \ge J(\u) + \frac \eps 2 \|u-\u\|_{L^2([0,T];\RR^N)}^2
	\end{align}
	and hence $\u$ is even a strict local minimizer of $J$ on the set $\UU$.
\end{thm}

\begin{proof}
	We will proceed analogously to the corresponding proof in \cite{casas-reyes-troeltzsch} and argue by contradiction. Therefore, we assume that $\u$ does not satisfy the quadratic growth condition \eqref{QGC}. Then, there exists a sequence $(u_k)_{k\in\NN}\subset \UU\setminus\{\u\}$ such that
	\begin{align}
	\label{ASS:JU}
	u_k \to \u \;\text{in}\;\DD \tand \forall k\in\NN:\;\; J(\u) + \frac 1 k \|u_k-\u\|_{\DD}^2 > J(u_k)
	\end{align}
	For $k\in\NN$, we define\vspace{-2mm}
	\begin{align*}
	d_k:=  \|u_k-\u\|_{\DD} \tand h_k:= \frac{1}{d_k}(u_k-\u).
	\end{align*}
	Since $\|h_k\|_{\DD}=1$, the Banach-Alaoglu theorem implies that there exists some function $h\in\DD$ such that $h_k\wto h$ in $\DD$ up to a subsequence. We split the argument into several steps:\pskip
	
	\textit{Step 1:} We will show that $J'(\u)[h] = 0$. The mean value theorem yields
	\begin{align*}
	J(u_k)  = J(\u) + d_k J'(v_k)[h_{k}]
	\end{align*}
	where $v_k$ is a point in $\DD$ between $\u$ and $u_k$ and hence
	\begin{align*}
	&J'(v_k)[h_{k}] = \frac 1 {d_k} \big( J(u_k) - J(\u) \big) 
	<\frac 1 {kd_k} \|u_k-\u\|_{\DD}^2 = \frac 1 {k} \|u-\u\|_{\DD}.
	\end{align*}
	Since $u_k \to u$ in $\DD$ we also have $v_k\to \u$ in $\DD$ as $v_k$ was chosen between $\u$ and $u_k$. Because of continuity, this implies that ${f_{v_k}\to f_\u}$ and ${g_{v_k}\to g_\u}$ in $C([0,T];C^1_b(\RR^6)\big)$. Since for general $u\in\UU$ and $h\in\DD$ the derivative $J'(u)[h]$ can be expressed by \eqref{FDJ} one can easily conclude that
	\begin{align*}
	J'(\u)[h] = \underset{k\to\infty}{\lim} \; J'(v_k)[h_k] \le \underset{k\to\infty}{\lim}\;  \frac 1 {k} \|u_k-\u\|_{\DD} = 0.
	\end{align*}
	On the other hand, it follows from condition (NC1) that
	\begin{align*}
	J'(\u)[h] = \frac 1 {d_k} \sum_{i=1}^N \int\limits_0^T \big( \lambda_i \u_i - p_i(\u) \big) (u_{k,i}-\u_i) \dt \ge 0
	\end{align*}
	and $J'(\u)[h] = 0$ immediately follows. \pskip
	
	\textit{Step 2:} We will show that $h\in\CC$. First note that the set
	\begin{align*}
	M_\u:=\left\{ h \in \DD \;\bigg|\; \begin{matrix} \text{For all}\; i\in\{1,...,N\} \;\text{and almost all}\; t\in[0,T]: \\[1mm] h_i(t) \ge 0 \;\text{if}\; \u_i(t) = a_i(t) \;\;\wedge\;\; h_i(t) \le 0 \;\text{if}\; \u_i(t) = b_i(t) \end{matrix} \right\} \subset \DD
	\end{align*}
	is obviously closed and convex. Thus it is weakly closed. One can easily see that for all $k\in\NN$, $u_k-\u$ belongs to $M_\u$ and hence the same holds for $h_k$. Consequently the weak limit $h$ must belong to $M_\u$ as well. Moreover, the result of Step 1 yields
	\begin{align}
	\label{EQ:LAGU}
	\sum_{i=1}^N \int\limits_0^T \big( \lambda_i \u_i(t) - p_i(\u)(t) \big)\; h_i(t) \dt = J'(\u)[h] = 0.
	\end{align}
	It is a well-known result that that the variational inequality (NC1) is equivalent to the so-called pointwise variational inequality (see \cite[p.\,68]{troeltzsch}). This means, it even holds that  $\big( \lambda_i \u_i(t) - p_i(\u)(t) \big)\; h_i(t) \ge 0$ for almost all $t\in [0,T]$ and all $i\in\{1,...,N\}$. Together with \eqref{EQ:LAGU} we can conclude that
	\begin{align*}
	\big( \lambda_i \u_i(t) - p_i(\u)(t) \big)\; h_i(t) = 0, \quad i=1,...,N
	\end{align*}
	for almost all $t\in[0,T]$ which directly yields $h_i(t)=0$ for almost all $t\in[0,T]$ with $\lambda_i \u_i(t) - p_i(\u)(t) \neq 0$. This proves $h\in\CC$.\pskip
	
	\textit{Step 3:} We will now prove that $h=0$. Because of \eqref{CIQ2}, it suffices to show that $J''(\u)[h,h] \le 0$.	
	By a second-order Taylor expansion we obtain that
	\begin{align*}
	J(u_k) &= J(\u) + d_k J'(\u)[h_k] +  \frac{d_k^2}{2}  J''(w_k)[h_k,h_k] \\
	&= J(\u) + d_k J'(\u)[h_k] +  \frac{d_k^2}{2} J''(\u)[h_k,h_k] + \frac{d_k^2}{2} \Big[ J''(w_k)[h_k,h_k] -J''(\u)[h_k,h_k] \Big] 
	\end{align*}
	for any $k\in\NN$ where $w_k$ is a point in $\DD$ between $u_k$ and $\u$. This means that
	\begin{align}
	\label{EQ:TA}
	J''(\u)[h_k,h_k] = \frac{2}{d_k^2}\Big[ J(u_k)-J(\u) \Big] - \frac{2}{d_k} J'(\u)[h_k] - \Big[ J''(w_k)[h_k,h_k] -J''(\u)[h_k,h_k] \Big].
	\end{align}
	First note that, according to \eqref{FDJ} and condition (NC1),
	\begin{align}
	\label{EST:T1}
	\frac{2}{d_k} J'(\u)[h_k] = \frac{2}{d_k^2} J'(\u)[u_k-\u] = \frac{2}{d_k^2} \sum_{i=1}^N \int\limits_0^T \big( \lambda_i \u_i - p_i(\u) \big) (u_{k,i}-\u_i) \dt \ge 0.
	\end{align}
	From \eqref{ASS:JU} it follows that
	\begin{align}
	\label{EST:T2}
	\frac{2}{d_k^2}\Big[ J(u_k)-J(\u) \Big] < \frac{2}{k d_k^2} \|u_k-\u\|_\DD^2 = \frac 2 k.
	\end{align}
	Since $\|h_k\|_\DD = 1$, Proposition \ref{TFDJ2}\,(b) yields
	\begin{align}
	\label{EST:T3}
	&\Big| J''(w_k)[h_k,h_k] -J''(\u)[h_k,h_k] \Big| 
	\le \|w_k-\u\|_\DD^\gamma \le d_k^\gamma.
	\end{align}
	Now, we can use \eqref{EST:T1}-\eqref{EST:T3} to bound the right-hand side of \eqref{EQ:TA} from above. We obtain
	\begin{align}
	\label{LIMSUP}
	J''(\u)[h_k,h_k] < \frac 2 k + d_k^\gamma \to 0, \;\; k\to\infty, \qquad\text{i.e.,}\quad \underset{k\to\infty}{\lim\sup}\; J''(\u)[h_k,h_k] \le 0.
	\end{align}
	From $h_k\wto h$ in $\DD$ we can also deduce that every component $h_{k,i}$ converges to $h_i$ weakly in $L^2([0,T])$. As the norm is weakly lower semicontinuous, it follows that 
	\begin{align}
	\|h_i\|_{L^2([0,T])} \le \underset{k\to\infty}{\lim\inf}\; \|h_{k,i}\|_{L^2([0,T])}, \quad i=1,...,N.
	\end{align}
	For any $i\in\{1,...,N\}$, we will now consider the bilinear functional
	\begin{align*}
	\mathcal B_i(h,\th) := \int\limits_{[0,T]\times\RR^6}\hspace{-3mm} (v\stimes m_i)\cdot \Big( \delv f_\u\, g_\u'[h] - \delv g_\u\; f_\u'[h] \Big) \, \th_i \dtxv, \quad h,\th\in\DD.
	\end{align*}
	Recall that the Fréchet derivatives $f_\u'[\cdot]$ and $g_\u'[\cdot]$ are linear and bounded operators. Thus, by the definition of weak convergence, $f_\u'[h_k] \to f_\u'[h]$ and $g_\u'[h_k]\to g_\u'[h]$ as $k\to\infty$. Using Hölder's inequality we obtain that $\mathcal B_i(h_k,h_k) - \mathcal B_i(h,h_k) \to 0,\; k\to\infty$. Moreover, $\mathcal B_i(h,h_k) - \mathcal B_i(h,h) \to 0$ for $k\to\infty$ follows directly from the weak convergence of $(h_k)$ in $\DD$. In summary, we have
	\begin{align*}
	\mathcal B_i(h_k,h_k) \to \mathcal B_i(h,h), \quad k\to\infty.
	\end{align*}
	Consequently, using \eqref{LIMSUP} we obtain that
	\begin{align*}
	J''(\u)[h,h] &= \sum_{i=1}^N \Big\{ \lambda_i \|h_i\|_{L^2([0,T])}^2 - \mathcal B_i(h,h) \Big\}  
	\le \sum_{i=1}^N \Big\{ \lambda_i \,  \underset{k\to\infty}{\lim\inf}\; \|h_{k,i}\|_{L^2([0,T])}^2 + \underset{k\to\infty}{\lim\inf}\; \big[-\mathcal B_i(h_k,h_k)\big] \Big\} \\[1mm]
	&\le \underset{k\to\infty}{\lim\sup}\; J''(\u)[h_k,h_k] \le 0.
	\end{align*}
	From \eqref{CIQ2} and the fact that $h\in\CC$ we can finally conclude that $h$ must be the null function. \pskip
	
	\textit{Step 4:} We will now show that $\|h_k\|_\DD$ converges to $h=0$ strongly in $\DD$. The bound $J''(\u)[h_k,h_k]\le 2/k + d_k^\gamma$ from \eqref{LIMSUP} particularly implies that
	\begin{align*}
	\sum_{i=1}^N \lambda_i \|h_{k,i}\|_{L^2([0,T])} \le \frac 2 k + d_k^\gamma + \sum_{i=1}^N \mathcal B_i(h_k,h_k) \to 0, \quad k\to\infty.
	\end{align*}
	Thus, $\|h_{k,i}\|_{L^2([0,T])} \to 0$ for all $i\in\{1,...,N\}$ which immediately yields $\|h_{k}\|_\DD \to 0$ as $k\to\infty$. \pskip
	
	However this is a contradiction, because $h_k$ was constructed in such a way that $\|h_k\|_\DD = 1$ for all $k\in\NN$. This means that our initial assumption was wrong and hence $\u$ does satisfy the quadratic growth condition \eqref{QGC}. The proof is complete.	
\end{proof}


\footnotesize

\end{document}